\pdfoutput=1
\documentclass[pdflatex,sn-mathphys-num]{sn-jnl}

\usepackage{amsmath,amssymb,amsfonts}
\usepackage{amsthm}
\usepackage{dsfont}
\usepackage{booktabs}

\theoremstyle{thmstyleone}
\newtheorem{theorem}{Theorem}
\newtheorem{lemma}[theorem]{Lemma}
\newtheorem{proposition}[theorem]{Proposition}
\newtheorem{corollary}[theorem]{Corollary}

\theoremstyle{thmstyletwo}
\newtheorem*{remark}{Remark}

\theoremstyle{thmstylethree}
\newtheorem*{axiom}{Axiom}

\DeclareMathOperator{\supp}{supp}
\DeclareMathOperator{\Var}{Var}
\DeclareMathOperator{\dist}{dist}
\DeclareMathOperator*{\esssup}{ess\,sup}

\begin{document}

\title[Hierarchical symmetry selects log-Poisson cascades]%
{Hierarchical symmetry selects log-Poisson cascades:
classification, uniqueness, and stability}

\author*[1]{\fnm{E. M.} \sur{Freeburg}}\email{hqops@icloud.com}

\affil*[1]{\orgname{Independent Researcher}}

\abstract{Within i.i.d.\ multiplicative cascades, a single axiom---the
hierarchical symmetry, a linear contraction on incremental scaling
exponents---is shown to be necessary and sufficient for the cascade
multiplier to be log-Poisson.  We prove: (1)~a characterization
theorem determining the log-Poisson law with explicit parameters,
within the class of \emph{all} multipliers with finite lattice
moments; (2)~a classification theorem locating the log-Poisson class
inside the log-infinitely-divisible family and identifying the
mechanism by which every rival sub-family fails the symmetry; (3)~a
stability theorem with sharp constants---$(1+\beta)^{1/2}$ when the
limiting increment is known, $\sqrt2$ when it is fitted---and (4)~an
unconditional propagation theorem transferring the bound to the
multiplier distribution at the sharp rate $\Theta(\sqrt\varepsilon)$,
with a matching lower bound.  Beyond independence, the classification
is shown to extend \emph{exactly} at the level of asymptotic
statistics (limiting cumulant generating function, large deviations,
multifractal spectrum) and \emph{provably not} at the level of laws:
an explicit stationary ergodic Markov multiplier satisfies the
symmetry exactly with a non-log-Poisson marginal, while exchangeable
multipliers collapse to the i.i.d.\ log-Poisson cascade and
finite-state Markov multipliers cannot satisfy the symmetry at all.
In the continuous category of exactly scale-invariant
log-infinitely-divisible multifractal random measures, no finite
moment window of structure-function exponents identifies the cascade
class, whereas at the level of the scale-invariance generator the
symmetry selects exactly the Barral--Mandelbrot compound Poisson
cascade, with scale-ratio-free stability constants.  The proofs
reduce to second-moment identities on $[0,1]$ via the change of
variables $u = e^{kx}$, boundedness of the multiplier
($\esssup W = r^\gamma$), and multiplicative couplings.}



\maketitle

\section{Introduction}

Multiplicative cascades model the successive fragmentation of a
conserved quantity across scales and arise in fully developed
turbulence \cite{K62,Man74}, rainfall, finance, and other settings
exhibiting intermittent, scale-invariant fluctuations.  The
mathematical foundations of multiplicative cascades were established
by Kahane and Peyri\`ere~\cite{KP76}; see Barral and
Mandelbrot~\cite{BM02} and Bacry and Muzy~\cite{BM03} for the modern
discrete and continuous theories.  The statistical properties of a
cascade are encoded in the scaling exponents $\zeta_p$ of the
structure functions $S_p(\ell) = \langle |\Phi(\ell)|^p \rangle \sim
\ell^{\zeta_p}$.  A central question is: \emph{which probability
distributions on the cascade multiplier $W$ are compatible with
observed scaling laws?}

Kolmogorov~\cite{K62} proposed log-normal multipliers, leading to
quadratic scaling exponents.  Z.-S.~She and
L\'ev\^eque~\cite{SL94} introduced a hierarchical symmetry for the
scaling exponents and derived a different, log-Poisson exponent
formula that has since shown excellent agreement with experimental
data.  Dubrulle~\cite{Dub94} independently identified the log-Poisson
form.  Z.-S.~She and Waymire~\cite{SW95} used the
L\'evy--Khintchine representation to argue that this symmetry selects
log-Poisson within the log-infinitely-divisible family; Dubrulle and
Graner~\cite{DG96} reached a similar conclusion via symmetry groups.
These works provided compelling physical arguments but did not supply
rigorous proofs.  Z.-S.~She and Zhang~\cite{SZ09} subsequently
proposed that the hierarchical symmetry is \emph{universal}---applicable
not only to turbulence but to general multi-scale fluctuation systems
including MHD turbulence, natural image statistics, and biological
signals---and should serve as a standard analytical framework.  The
present paper supplies the rigorous mathematical foundation for this
program.

We formalize the hierarchical symmetry as a single axiom~(A1) and
prove the following.

\emph{Characterization}
(Theorem~\ref{thm:characterization}).  A1 uniquely determines the
cascade multiplier $W$ to be log-Poisson, with parameters
$(\,a = \gamma\ln r,\; b = (\ln\beta)/k,\; \lambda = -C\ln r\,)$
expressed in terms of the observable scaling exponents.  No other
distribution---infinitely divisible or not---is compatible with~A1.

\emph{Classification}
(Theorem~\ref{thm:classification}).  Within the full
log-infinitely-divisible family, A1 selects exactly the log-Poisson
class, and the proof stratifies the exclusion: Gaussian components,
positive jumps, and stable generators of index $\alpha \in [1,2)$ are
eliminated by divergence of the incremental exponents; all remaining
generators are eliminated by a second-moment rigidity.

\emph{Stability}
(Theorem~\ref{thm:stability}).  If A1 holds only approximately, with
residuals bounded by~$\varepsilon$, then the (tilted, compactified)
L\'evy measure of the generator is within $K\sqrt{\varepsilon}$ of a
Dirac mass in the Wasserstein-1 metric, with the sharp constant
$K = ((1+\beta)|\ln r|/|A|)^{1/2}$ when the limiting increment
$\delta_\infty$ is known, and the sharp constant
$(2|\ln r|/|A|)^{1/2}$ when $\delta_\infty$ is fitted along with
$\beta$.

\emph{Propagation at the sharp rate}
(Theorems~\ref{thm:propagation} and~\ref{thm:lower}).  Closeness of
the L\'evy measure transfers to closeness of the multiplier
distribution at rate $O(\sqrt\varepsilon)$ \emph{unconditionally}---%
no finite-activity or minimum-jump hypothesis---with an explicit
constant; and an explicit family shows the rate
$\Theta(\sqrt\varepsilon)$ is exact.

\emph{Beyond independence}
(Section~\ref{sec:beyond}).  For stationary ergodic multipliers, A1
is equivalent to the limiting cumulant generating function---hence
the large-deviation rate and the multifractal spectrum---being
exactly log-Poisson (Theorem~\ref{thm:cgflevel}); but it provably
does not determine the multiplier law: an explicit stationary
ergodic countable-state Markov multiplier satisfies A1 exactly with
a non-log-Poisson marginal (Theorem~\ref{thm:interleaved}).
Exchangeable multipliers with exact scaling collapse to the i.i.d.\
log-Poisson cascade (Corollary~\ref{cor:exch}), and finite-state
Markov multipliers cannot satisfy A1 at all with nontrivial
intermittency (Theorem~\ref{thm:finitestate}).

\emph{Continuous cascades}
(Section~\ref{sec:mrm}).  In the Bacry--Muzy category of exactly
scale-invariant log-infinitely-divisible multifractal random
measures, structure-function exponents exist only on a finite moment
window, and \emph{no finite window identifies the cascade class}
(Theorem~\ref{thm:window}); at the level of the scale-invariance
generator---magnitude statistics, observable at all orders---A1
selects exactly the Barral--Mandelbrot compound Poisson cascade
(Theorem~\ref{thm:cpc}), with stability constants that are native
and scale-ratio-free (Corollary~\ref{cor:mrmstab}).

The converse---that log-Poisson multipliers imply
A1---is established in Proposition~\ref{prop:converse}, yielding a
biconditional equivalence (Corollary~\ref{cor:biconditional}).

\smallskip
\noindent\textbf{Relation to prior work.}
The exponent formula~\eqref{eq:zeta} (Lemma~\ref{lem:exponent})
and the log-Poisson identification are due to
Z.-S.~She and L\'ev\^eque~\cite{SL94} and Dubrulle~\cite{Dub94}.
Z.-S.~She and Waymire~\cite{SW95} gave the first argument connecting
A1 to the L\'evy--Khintchine classification.  The compound Poisson
cascades appear in Barral--Mandelbrot~\cite{BM02}; the log-ID
multifractal random measures in Bacry--Muzy~\cite{BM03,MB02};
magnitude-cumulant analysis in Delour--Muzy--Arneodo~\cite{DMA01}.
The following results are new: the boundedness and
moment-determinacy lemma (Lemma~\ref{lem:determinacy}); the converse
and biconditional (Proposition~\ref{prop:converse},
Corollary~\ref{cor:biconditional}); the classification with
stratified exclusion (Theorem~\ref{thm:classification}); the
determinacy dichotomy (Proposition~\ref{prop:determinacy}); the
stability theorem with sharp constants under both readings
(Theorem~\ref{thm:stability}); the unconditional sharp-rate
propagation theory (Theorems~\ref{thm:propagation},
\ref{thm:lower}); the beyond-i.i.d.\ classification and its
impossibility boundary (Theorems~\ref{thm:cgflevel},
\ref{thm:interleaved}, \ref{thm:finitestate},
Corollary~\ref{cor:exch}); and the continuous-cascade results
(Theorems~\ref{thm:window}, \ref{thm:cpc}).

\smallskip
\noindent\textbf{Method.}
The change of variables $u = e^{kx}$ maps the L\'evy measure from
$(-\infty,0]$ to the compact interval $[0,1]$, where a second-moment
identity against the candidate atom decides the classification and
its stability.  A1 forces the multiplier to be essentially bounded
with $\esssup W = r^\gamma$, which yields moment determinacy
directly and eliminates all tail estimates from the propagation
argument; the propagation itself is a multiplicative coupling in
which jumps near the identity are costed by their multiplicative
deviation rather than counted.  Beyond independence the arguments
run at the level of limiting cumulant generating functions; in the
continuous category they transfer through the dictionary
$\ln r \mapsto -1$.

\smallskip
Throughout this paper, $\log$ denotes the natural logarithm.

\section{Setup}
\label{sec:setup}

Let $r \in (0,1)$ be a scale ratio.  A \emph{multiplicative cascade}
generates a random positive measure $\mu$ on nested sets $B_0 \supset
B_1 \supset \cdots$ via
\[
  \mu(B_{n+1}) = W_{n+1} \cdot \mu(B_n),
\]
where $\{W_n\}$ are i.i.d.\ positive random variables with
$\mathbb{E}[W] = 1$ (conservation of mean; a normalization
convention---none of the proofs below depend on it, cf.\
Corollary~\ref{cor:conservation}).  Sections~\ref{sec:beyond}
and~\ref{sec:mrm} relax, respectively, the independence assumption
and the discrete setting.

Let $\Phi(\ell)$ be the cascade observable at scale $\ell = r^n$.
Define the \emph{structure functions}
\[
  S_p(\ell) = \langle |\Phi(\ell)|^p \rangle = \ell^{\zeta_p},
\]
where $\zeta_p$ are the \emph{scaling exponents}, with $\zeta_0 = 0$;
all moments of $W$ are assumed finite, so that $\zeta_p$ is finite
for every $p \ge 0$.  The relation $\mathbb{E}[W^p] = r^{\zeta_p}$
identifies the single-step moment structure of the multiplier with
the observable's scaling.

For a fixed integer $k \geq 1$ (the \emph{hierarchy step}), define
the \emph{moment ratios}
\[
  H_p(\ell) = \frac{S_{p+k}(\ell)}{S_p(\ell)} = \ell^{\delta_p},
\]
where $\delta_p = \zeta_{p+k} - \zeta_p$ are the \emph{incremental
exponents} at step~$k$.

\section{The axiom}
\label{sec:axiom}

\begin{axiom}[A1: Hierarchical Symmetry]
There exist $\beta \in (0,1)$ and $L \in \mathbb{R}$ such that for
all $p \in k\mathbb{N}_0$ (non-negative integer multiples of~$k$),
the incremental exponents satisfy
\begin{align}\label{eq:A1}
  \delta_{p+k}
  = (1 - \beta)\,L + \beta\,\delta_p.
  \tag{$*$}
\end{align}
\end{axiom}

Since \eqref{eq:A1} is a contraction, it forces
$\delta_{mk} \to L$ as $m \to \infty$; we write
$\delta_\infty := L$.  (Stating the axiom with a free constant $L$,
rather than defining $\delta_\infty$ as a limit inside the equation
that uses it, matters only for the approximate version in
Theorem~\ref{thm:stability}, where the distinction between the true
limit and a fitted constant is quantitatively significant.)

A1 determines the full parameter set from the observable exponents:
\begin{center}\small
\begin{tabular}{@{}lll@{}}
\toprule
Parameter & Determined by & Meaning \\
\midrule
$\beta$ & Contraction ratio of~\eqref{eq:A1} & Coupling strength \\
$\gamma$ & $\delta_\infty / k$ & Linear drift \\
$C$ & $(\delta_0 - \delta_\infty)/(1 - \beta)$ & Concentration amplitude \\
\bottomrule
\end{tabular}
\end{center}

\smallskip\noindent\textit{Edge case} ($C = 0$).  If $\delta_0 =
\delta_\infty$, then $C = 0$ and $\zeta_p = \gamma p$ (monofractal
scaling).  The cascade multiplier $W = r^\gamma$ is deterministic.
A1 is trivially satisfied for any $\beta \in (0,1)$.  The
classification and stability theorems assume $C > 0$ (nontrivial
intermittency).

\section{Results: the i.i.d.\ cascade}
\label{sec:results}

\begin{lemma}[Exponent Form]\label{lem:exponent}
If the incremental exponents $\{\delta_p\}$ satisfy the A1
recurrence~\eqref{eq:A1} with $\beta \in (0,1)$, then with $\zeta_0
= 0$:
\begin{align}\label{eq:zeta}
  \zeta_p = \gamma p + C\bigl(1 - \beta^{p/k}\bigr),
  \tag{$**$}
\end{align}
where $\gamma = \delta_\infty / k$ and $C = (\delta_0 -
\delta_\infty)/(1 - \beta)$.

We emphasize that this lemma is purely algebraic and involves no
probabilistic content.
\end{lemma}

\begin{proof}
(1)~The recurrence~\eqref{eq:A1} is first-order linear with fixed
point $\delta_\infty$:
\[
  \delta_{p+k} - \delta_\infty = \beta\,(\delta_p - \delta_\infty).
\]

(2)~At $p = mk$ (integer multiples of~$k$), iteration gives
\[
  \delta_{mk} - \delta_\infty
  = (\delta_0 - \delta_\infty)\,\beta^m.
\]
Replacing $m = p/k$:
\[
  \delta_p
  = \delta_\infty + (\delta_0 - \delta_\infty)\,\beta^{p/k}.
\]
(This formula is derived at $p \in k\mathbb{N}_0$.  For general $p
\geq 0$, we define $\zeta_p$ by~\eqref{eq:zeta}; the function
$\beta^{p/k}$ is well-defined for all real $p \geq 0$ since $\beta >
0$.  The moment-based arguments in Lemma~\ref{lem:determinacy} and
Theorem~\ref{thm:characterization} use only the lattice
$p \in k\mathbb{N}_0$, where the formula is proved.)

(3)~With $\zeta_0 = 0$, sum the step-$k$ increments using the
geometric series:
\[
  \zeta_p
  = \frac{p}{k}\,\delta_\infty
    + \frac{\delta_0 - \delta_\infty}{1 - \beta}\,
      \bigl(1 - \beta^{p/k}\bigr).
\]
Identifying $\gamma = \delta_\infty/k$ and $C = (\delta_0 -
\delta_\infty)/(1-\beta)$:
\[
  \zeta_p = \gamma p + C\bigl(1 - \beta^{p/k}\bigr).  \qedhere
\]
\end{proof}

\begin{lemma}[Boundedness and Moment Determinacy]\label{lem:determinacy}
Let $\{W_n\}$ be an i.i.d.\ multiplicative cascade whose lattice
scaling exponents satisfy
$\zeta_{km} = \gamma km + C(1-\beta^{m})$ for all $m \in
\mathbb{N}_0$, as produced by A1 via Lemma~\ref{lem:exponent}.  Then:

\textup{(i)} $W$ is essentially bounded, with $\esssup W =
r^{\gamma}$;

\textup{(ii)} the law of $W$ is uniquely determined by the lattice
moments $\{\mathbb{E}[W^{km}]\}_{m \ge 0}$.
\end{lemma}

\begin{proof}
Set $V = W^k \ge 0$.  By independence across cascade levels,
$\mathbb{E}[(W_1\cdots W_n)^p] = (\mathbb{E}[W^p])^n$; combined with
$S_p(\ell) = r^{n\zeta_p}$ this gives, for a single step,
$\mathbb{E}[V^m] = \mathbb{E}[W^{km}] = r^{\zeta_{km}}$ for every
$m \in \mathbb{N}_0$---these are exactly the moments constrained
by~A1.  Then
\[
  \bigl(\mathbb{E}[V^m]\bigr)^{1/m}
  \;=\; r^{\zeta_{km}/m}
  \;=\; r^{\,\gamma k + C(1-\beta^m)/m}
  \;\longrightarrow\; r^{\gamma k}
  \qquad (m \to \infty).
\]
On a probability space the norms $\|V\|_{L^m}$ are nondecreasing in
$m$ and converge to $\|V\|_{L^\infty}$; hence $V$ is essentially
bounded with $\|V\|_\infty = r^{\gamma k}$, and $W = V^{1/k}$ is
bounded with $\|W\|_\infty = r^{\gamma}$, proving~(i).

For~(ii): a probability law supported in the compact interval
$[0, r^{\gamma k}]$ is uniquely determined by its integer moments.
(Hausdorff moment problem: polynomials are uniformly dense in
$C([0,r^{\gamma k}])$ by the Weierstrass approximation theorem, so
two laws with equal moments integrate every continuous function
equally and coincide by the Riesz representation theorem.)  Hence
the law of $V$ is determined by $\{\mathbb{E}[V^m]\}$; since
$x \mapsto x^{1/k}$ is a Borel bijection of $[0,\infty)$, the law of
$W$ is determined as well.
\end{proof}

\begin{remark}
A Carleman-condition argument \cite{Akh65} is available here only
along the lattice (applied to $V = W^k$), since A1 constrains no
other moments when $k \ge 2$.  The boundedness route above is
shorter and stronger: statement~(i) identifies the essential
supremum of the multiplier with the most-singular scaling factor
$r^\gamma$, a fact used again in Theorem~\ref{thm:propagation}.
\end{remark}

\begin{theorem}[Characterization]\label{thm:characterization}
Let $\{W_n\}$ be an i.i.d.\ multiplicative cascade whose incremental
scaling exponents satisfy A1.  Then:

\textup{(i)} \textbf{Scaling exponents.}
\[
  \zeta_p = \gamma p + C\bigl(1 - \beta^{p/k}\bigr),
\]
where $\gamma = \delta_\infty/k$ and $C = (\delta_0 -
\delta_\infty)/(1-\beta)$.

\textup{(ii)} \textbf{Uniqueness.}  The cascade multiplier $W$ is
uniquely determined to be log-Poisson:
\[
  \log W = a + bN, \qquad N \sim \mathrm{Poisson}(\lambda),
\]
\[
  a = \gamma \ln r, \quad
  b = \frac{\ln\beta}{k}, \quad
  \lambda = -C\ln r.
\]
No other probability distribution on~$W$ is compatible with A1.

\textup{(iii)} \textbf{Multifractal spectrum.}  Let $d$ denote the
spatial dimension of the cascade support.  Then
\[
  f(h) = d - C + Cx(1 - \ln x),
  \qquad
  x = \frac{k(h - \gamma)}{C|\ln\beta|},
\]
defined for $h \in [\gamma,\; \gamma + (C/k)|\ln\beta|]$.
\end{theorem}

\begin{proof}
\textit{(i)} Immediate from Lemma~\ref{lem:exponent}.

\textit{(ii)} Set $a = \gamma\ln r$, $b = (\ln\beta)/k$, $\lambda =
-C\ln r > 0$.  For $\log W = a + bN$ with $N \sim
\mathrm{Poisson}(\lambda)$:
\[
  \mathbb{E}[W^p] = e^{ap} \cdot \exp\bigl[\lambda(e^{bp} - 1)\bigr]
  \qquad\text{for all real } p \ge 0,
\]
and since $e^{bp} = \beta^{p/k}$, the choice $\lambda = -C\ln r$
gives $\lambda(\beta^{p/k}-1) = C\ln r\,(1-\beta^{p/k})$, so that
\[
  \mathbb{E}[W^{km}]
  = e^{(\gamma\ln r) km}\, e^{C\ln r\,(1-\beta^{m})}
  = r^{\zeta_{km}}
  \qquad\text{for every } m \in \mathbb{N}_0 .
\]
Thus the log-Poisson law realizes exactly the lattice moments of the
cascade multiplier.  By Lemma~\ref{lem:determinacy}(ii) the lattice
moments uniquely determine the law; hence $W$ is log-Poisson, and no
other distribution is possible.  (The same matching holds at every
real $p \ge 0$, so the extension of~\eqref{eq:zeta} off the lattice
is consistent.)

\textit{(iii)}  The singularity spectrum $f(h) = \inf_p\,[ph -
\zeta_p + d]$.  Setting the derivative to zero:
\[
  h - \gamma + \frac{C}{k}(\ln\beta)\,\beta^{p/k} = 0
  \quad\Longrightarrow\quad
  h - \gamma = \frac{C|\ln\beta|}{k}\,\beta^{p/k}.
\]
Define $x = \beta^{p/k} = k(h-\gamma)/(C|\ln\beta|)$.  Then $p =
-k\ln x / |\ln\beta|$ and
\[
  p(h - \gamma) = \frac{-k\ln x}{|\ln\beta|}
    \cdot \frac{C|\ln\beta|}{k}\,x = -Cx\ln x.
\]
Therefore
\[
  f(h) = d - C + Cx(1 - \ln x), \qquad
  x = \frac{k(h-\gamma)}{C|\ln\beta|}.
\]
Boundary checks: $p = 0 \Rightarrow x = 1 \Rightarrow f = d$; $p \to
\infty \Rightarrow x \to 0 \Rightarrow f \to d - C$.  Concavity:
$\zeta_p'' = -(C/k^2)(\ln\beta)^2\beta^{p/k} < 0$.
\end{proof}

\begin{corollary}[Most-singular branch]\label{cor:branch}
Under A1, $W \le r^\gamma$ almost surely, and the bound is attained
with positive probability: $\mathbb{P}(W = r^\gamma) = e^{-\lambda}$.
The probability that a cascade trajectory takes the maximal factor
for $n$ consecutive levels is
\[
  e^{-\lambda n} \;=\; r^{Cn} \;=\; \ell^{\,C}
  \qquad\text{at scale } \ell = r^n :
\]
the set of always-maximal cascade paths carries codimension exactly
$C$, in agreement with $f(h_{\min}) = d - C$ in
Theorem~\textup{\ref{thm:characterization}(iii)}.
\end{corollary}

\begin{proof}
By Theorem~\ref{thm:characterization}(ii), $W = r^\gamma
\beta^{N/k}$ with $N \sim \mathrm{Poisson}(\lambda)$, so $W \le
r^\gamma$ with equality iff $N = 0$, an event of probability
$e^{-\lambda}$.  By independence across levels, $n$ consecutive
maximal factors have probability $e^{-\lambda n} =
e^{(C\ln r) n} = r^{Cn}$.
\end{proof}

\begin{remark}[Scope of Theorem~\ref{thm:characterization}]
No infinite-divisibility assumption enters
Theorem~\ref{thm:characterization}: A1 characterizes the log-Poisson
law within the class of \emph{all} nonnegative multipliers with
finite lattice moments.  Theorem~\ref{thm:classification} below is
therefore not a larger uniqueness statement but an anatomical one:
it locates the log-Poisson class inside the L\'evy--Khintchine
parameterization and identifies the mechanism by which each rival
sub-family fails A1.  Its proof technique---the compactifying
substitution $u = e^{kx}$---is also the engine of the stability and
propagation theory.
\end{remark}

\begin{remark}[Conservation]
The setup assumes $\mathbb{E}[W] = 1$, which requires $\zeta_1 = 0$.
Substituting into~\eqref{eq:zeta}: $\gamma + C(1 - \beta^{1/k}) =
0$, giving $\gamma = -C(1-\beta^{1/k})$.  This is a constraint
relating $\gamma$ to~$C$ and~$\beta$, reducing the free parameters
from three to two.
\end{remark}

\begin{proposition}[Converse]\label{prop:converse}
If the cascade multiplier $W$ is log-Poisson---\allowbreak that is,
$\log W = a + bN$ with $N \sim \mathrm{Poisson}(\lambda)$, $b < 0$, $\lambda >
0$---then the incremental scaling exponents satisfy A1 with $\beta =
e^{bk} \in (0,1)$.
\end{proposition}

\begin{proof}
The moment generating function gives $\mathbb{E}[W^p] = \exp(ap +
\lambda(e^{bp} - 1))$, so $\zeta_p = \bigl(ap + \lambda(e^{bp} -
1)\bigr)/\ln r$.  The step-$k$ increments are
\[
  \delta_p = \frac{ak + \lambda e^{bp}(e^{bk}-1)}{\ln r}.
\]
Setting $\beta = e^{bk} \in (0,1)$ (since $b < 0$, $k \geq 1$):
\[
  \delta_p = \frac{ak}{\ln r}
  + \frac{\lambda(\beta - 1)}{\ln r}\,\beta^{p/k}.
\]
As $p \to \infty$: $\beta^{p/k} \to 0$, so $\delta_\infty = ak/\ln
r$.  The deviation is
\[
  \delta_p - \delta_\infty
  = \frac{\lambda(\beta - 1)}{\ln r}\,\beta^{p/k}.
\]
At $p + k$:
\[
  \delta_{p+k} - \delta_\infty
  = \frac{\lambda(\beta-1)}{\ln r}\,\beta^{(p+k)/k}
  = \beta\,(\delta_p - \delta_\infty).
\]
Therefore $\delta_{p+k} = (1-\beta)\delta_\infty + \beta\delta_p$,
which is exactly A1.
\end{proof}

\begin{corollary}[Biconditional]\label{cor:biconditional}
Within i.i.d.\ multiplicative cascades, A1 is necessary and
sufficient for log-Poisson:
\[
  \text{A1 holds}
  \;\;\Longleftrightarrow\;\;
  W \text{ is log-Poisson (with } b < 0\text{).}
\]
The forward direction is
Theorem~\textup{\ref{thm:characterization}(ii)}; the reverse is
Proposition~\textup{\ref{prop:converse}}.
\end{corollary}

\begin{theorem}[Log-ID Classification]\label{thm:classification}
Let $\{W_n\}$ be an i.i.d.\ multiplicative cascade with nontrivial
intermittency ($C > 0$), whose generator $\log W$ is infinitely
divisible with L\'evy triplet $(a, \sigma^2, \nu)$.  Then A1 holds
with $\beta \in (0,1)$ if and only if $\sigma^2 = 0$ and $\nu =
\lambda\delta_b$ for some $b < 0$, $\lambda > 0$.  That is:
\begin{center}
\emph{A1 selects exactly the log-Poisson class from the full
log-infinitely-divisible family.}
\end{center}
No other log-ID cascade---log-normal, log-stable, or any
intermediate---satisfies A1.
\end{theorem}

\begin{proof}
\textit{Reverse direction.} If $\nu = \lambda\delta_b$ with $b < 0$
and $\sigma^2 = 0$, then $\log W = a + bN$ with $N \sim
\mathrm{Poisson}(\lambda)$, and A1 holds by
Proposition~\ref{prop:converse}.

\textit{Forward direction.}  Assume A1 holds.  We show $\sigma^2 = 0$
and $\nu = \lambda\delta_b$.

\textit{Step~1 (unsplit form).}  The cumulant generating function of
$\log W$ is
\[
  \psi(p) = ap + \frac{\sigma^2 p^2}{2}
  + \int\bigl(e^{px} - 1 - px\,\mathds{1}_{|x|\leq 1}\bigr)
  \,\nu(dx),
\]
finite for all $p \ge 0$ since all moments of $W$ are finite.  With
$\zeta_p = \psi(p)/\ln r$ and $\delta_p = (\psi(p+k) -
\psi(p))/\ln r$, define $\phi(p) = \psi(p+k) - \psi(p)$.  Then
\[
\begin{aligned}
  \phi(p) &= ak + \sigma^2 k\!\left(p + \tfrac{k}{2}\right)
  + \int g_p(x)\,\nu(dx),\\
  g_p(x) &:= e^{px}\bigl(e^{kx}-1\bigr) - kx\,\mathds{1}_{|x|\leq 1},
\end{aligned}
\]
where $g_p$ is $\nu$-integrable for each $p$, being the difference of
the two compensated L\'evy--Khintchine integrands.  \emph{No
splitting of the integral is performed at this stage.}

\textit{Step~1$'$ (sign inventory).}  For every $p \ge 0$:
\begin{itemize}
\item on $(0,1]$: $e^{px} \ge 1$ gives $g_p(x) \ge (e^{kx}-1) - kx
  \ge 0$, and $g_p(x) \uparrow \infty$ pointwise as $p \to \infty$;
\item on $(1,\infty)$: $g_p(x) = e^{px}(e^{kx}-1) \ge 0$, increasing
  to $+\infty$ pointwise;
\item on $[-1,0)$: $|e^{px}(e^{kx}-1)| \le 1 - e^{kx} \le k|x|$, so
  $0 \le g_p(x) \le k|x|$, with $g_p(x) \to k|x|$ pointwise as
  $p \to \infty$;
\item on $(-\infty,-1)$: $-1 \le g_p(x) \le 0$, with $g_p(x) \to 0$
  pointwise.
\end{itemize}
In particular $\int g_p\,d\nu \ge -\nu((-\infty,-1))$, uniformly in
$p$.

\textit{Step~2} ($\sigma^2 = 0$).  If $\sigma^2 > 0$ then, by
Step~1$'$,
\[
  \phi(p) \;\ge\; ak + \sigma^2 k\Bigl(p + \tfrac{k}{2}\Bigr)
  - \nu\bigl((-\infty,-1)\bigr) \;\longrightarrow\; +\infty,
\]
so $\delta_p = \phi(p)/\ln r \to -\infty$, contradicting the finite
limit $\delta_\infty$ forced by A1.  Hence $\sigma^2 = 0$.
\textit{This eliminates all log-normal and mixed Gaussian-jump
generators.}

\textit{Step~3} ($\supp(\nu) \subseteq (-\infty,0]$).  If $\nu$ has
mass on $(0,\infty)$ then, since $g_p \ge 0$ there and $g_p \uparrow
\infty$ pointwise, monotone convergence gives $\int_{(0,\infty)}
g_p\,d\nu \to \infty$, while $\int_{(-\infty,0)} g_p\,d\nu \ge
-\nu((-\infty,-1))$; again $\phi(p) \to \infty$ and $\delta_p \to
-\infty$, a contradiction.  Therefore $\supp(\nu) \subseteq
(-\infty,0]$.  \textit{This eliminates all generators with positive
jumps.}

\textit{Step~3\textonehalf} (integrability near~$0$).  We claim A1
forces $\int_{[-1,0)} |x|\,\nu(dx) < \infty$.  With $\sigma^2 = 0$
and $\supp\nu \subseteq (-\infty,0]$, apply Fatou's lemma on
$[-1,0)$ (integrand $g_p \ge 0$ by Step~1$'$, pointwise limit
$k|x|$) and dominated convergence on $(-\infty,-1)$ (bounded by~$1$,
finite mass):
\[
  \liminf_{p\to\infty} \phi(p)
  \;\ge\; ak + k\int_{[-1,0)} |x|\,\nu(dx)
  - \nu\bigl((-\infty,-1)\bigr).
\]
If $\int_{[-1,0)}|x|\,d\nu = \infty$ then $\phi(p) \to \infty$ and
$\delta_p \to -\infty$, contradicting A1.  Hence
$\int_{|x|\le1}|x|\,d\nu < \infty$, the compensator integral
$k\int x\,\mathds{1}_{|x|\le1}\,d\nu$ is finite, and \emph{only now}
may the integral be split:
\[
  \phi(p) = c_0 + \int_{(-\infty,0)} e^{px}\bigl(e^{kx}-1\bigr)
  \,\nu(dx),
  \qquad
  c_0 = ak - k\!\int x\,\mathds{1}_{|x|\leq 1}\,\nu(dx).
\]
The remaining integrand is dominated by $1 - e^{kx} \le
\min(1, k|x|) \in L^1(\nu)$ and tends to $0$ pointwise, so by
dominated convergence $\phi(p) \to c_0$ as $p \to \infty$; thus
$\phi_\infty = c_0$ and $\delta_\infty = c_0/\ln r$.

\textit{Step~4} ($\nu$ is a single Dirac mass).  A1 at $p = mk$
gives, by Lemma~\ref{lem:exponent}(2),
$\delta_{mk} - \delta_\infty = (\delta_0-\delta_\infty)\beta^m$;
multiplying by $\ln r$,
\begin{equation}\label{eq:dagger}
  \int_{(-\infty,0)} e^{mkx}\bigl(e^{kx}-1\bigr)\,\nu(dx)
  = A\beta^m
  \qquad\text{for all } m \geq 0,
\end{equation}
where $A = (\delta_0 - \delta_\infty)\ln r < 0$ (nontrivial
intermittency and $\ln r < 0$).  Substitute $u = e^{kx}$, mapping
$(-\infty,0) \to (0,1)$; let $\tilde\nu$ be the pushforward of $\nu$
and set
\[
  \eta := (1-u)\,d\tilde\nu \;\ge\; 0,
  \qquad
  \mu_m := \int_{(0,1)} u^m \, d\eta .
\]
Then \eqref{eq:dagger} reads $\mu_m = |A|\,\beta^m$ for all $m \ge
0$; the case $m=0$ shows $\eta$ is a finite positive measure of
total mass $|A|$.  Only $m \in \{0,1,2\}$ are needed:
\[
  \int_{(0,1)} (u-\beta)^2 \, d\eta
  \;=\; \mu_2 - 2\beta\mu_1 + \beta^2\mu_0
  \;=\; |A|\bigl(\beta^2 - 2\beta^2 + \beta^2\bigr)
  \;=\; 0 .
\]
Since $(u-\beta)^2 > 0$ on $(0,1)\setminus\{\beta\}$ and $\eta \ge
0$, we conclude $\eta\bigl((0,1)\setminus\{\beta\}\bigr) = 0$ and
$\eta(\{\beta\}) = |A|$.  Because $u - 1 \neq 0$ on $(0,1)$, the
tilt is invertible:
\[
  \tilde\nu = \frac{|A|}{1-\beta}\,\delta_\beta
  = \lambda\,\delta_\beta,
  \qquad \lambda = \frac{|A|}{1-\beta} > 0 .
\]
Therefore $\nu = \lambda\delta_b$ with $b = (\ln\beta)/k < 0$ and
$\lambda > 0$.  The generator $\log W$ is compound Poisson with
deterministic jump size~$b$ and rate~$\lambda$: this is the
log-Poisson distribution.
\end{proof}

\begin{remark}[Alternative identification; minimality]
The conclusion of Step~4 can also be reached from the full moment
sequence: a finite signed measure on a compact interval is determined
by its moments (Weierstrass approximation and the Riesz
representation theorem), and $A\delta_\beta$ realizes the
moments~\eqref{eq:dagger}.  The second-moment argument given above is
preferred because it (a)~uses only $m \in \{0,1,2\}$
of~\eqref{eq:dagger}, so that A1 restricted to $p \in \{0, k, 2k\}$,
together with Steps~2--3\textonehalf, already pins the distribution;
and (b)~is exactly the computation that the stability theorem
quantifies (see the Remark closing
Section~\ref{sec:propagation}).
\end{remark}

\begin{remark}[Which families die where]
The exclusion mechanism stratifies.  Gaussian components (Step~2),
positive jumps (Step~3), and negative-support L\'evy measures with
$\int_{|x|\le1}|x|\,d\nu = \infty$---in particular totally skewed
stable generators of index $\alpha \in [1,2)$---all fail A1 by
\emph{divergence}: $\delta_\infty = -\infty$
(Step~3\textonehalf).  All remaining log-ID generators have bounded
incremental exponents but fail the \emph{geometric rigidity} of
Step~4: e.g.\ for a stable generator of index $\alpha < 1$ the
moments $\mu_m$ decay like the power law $m^{\alpha-1}$, which cannot
equal $|A|\beta^m$ for any $\beta \in (0,1)$.
\end{remark}

\begin{corollary}[Principal cascade classes]\label{cor:partition}
The log-ID cascade family is partitioned by A1:
\begin{center}\footnotesize\setlength{\tabcolsep}{3pt}
\begin{tabular}{@{}lllll@{}}
\toprule
Class & L\'evy data & A1 & Failure mode & Determinate \\
\midrule
Log-Poisson & \begin{tabular}[t]{@{}l@{}}$\sigma^2=0$,\\
  $\nu=\lambda\delta_b$, $b<0$\end{tabular} & Holds
  & --- & Yes (bounded $W$) \\
\addlinespace
Log-normal & $\sigma^2>0$ &
  Fails & \begin{tabular}[t]{@{}l@{}}$\delta_\infty=-\infty$\\
  (Step 2)\end{tabular} & No \\
\addlinespace
Log-stable, $\alpha\in[1,2)$ & $\nu$ power-law &
  Fails & \begin{tabular}[t]{@{}l@{}}$\delta_\infty=-\infty$\\
  (Step 3\textonehalf)\end{tabular} & --- \\
\addlinespace
Log-stable, $\alpha<1$ & $\nu$ power-law &
  Fails & \begin{tabular}[t]{@{}l@{}}non-geometric\\
  decay (Step 4)\end{tabular} & Yes (bounded $W$) \\
\addlinespace
General log-ID & any other & Fails & Step 3 or Step 4 & --- \\
\bottomrule
\end{tabular}
\end{center}
Determinacy in this family tracks boundedness of the multiplier,
equivalently boundedness of $\{\delta_p\}$
(Lemma~\ref{lem:determinacy}(i)): every negative-support generator
with finite $\delta_\infty$ has compactly supported $W$, hence is
moment-determinate.  The operative dichotomy is bounded versus
unbounded, with A1 strictly on the bounded side.
\end{corollary}

\begin{proposition}[Determinacy Dichotomy]\label{prop:determinacy}
The two principal cascade exponent laws are distinguished by moment
determinacy:

\textup{(a)} If A1 holds within a cascade (log-Poisson regime), then
the scaling exponents uniquely determine the multiplier law
(Theorem~\textup{\ref{thm:characterization}(ii)}); indeed $W$ is
bounded and moment-determinate
(Lemma~\textup{\ref{lem:determinacy}}).

\textup{(b)} If the exponents are quadratic, $\zeta_p = c_1 p +
c_2 p^2$ with $c_2 < 0$ (\textup{\cite{K62}}/log-normal regime),
then $\mathbb{E}[W^p] = \exp(\mu p + \sigma^2 p^2/2)$ with $\mu =
c_1\ln r$ and $\sigma^2 = 2c_2\ln r > 0$, and this moment sequence is
indeterminate: uncountably many distinct laws realize it.  Under
quadratic scaling, the exponents cannot identify the multiplier law.
\end{proposition}

\begin{proof}
\textit{Part~\textup{(a)}} is contained in
Lemma~\ref{lem:determinacy} and
Theorem~\ref{thm:characterization}(ii).

\textit{Part~\textup{(b)}.}  We exhibit the family (Heyde
\cite{Hey63}).  Let $f$ be the log-normal density with parameters
$(\mu, \sigma^2)$ and, for $|c| \le 1$, define
\[
  f_c(x) \;=\; f(x)\,\Bigl[\,1 + c\,
  \sin\!\Bigl(\tfrac{2\pi(\ln x - \mu)}{\sigma^2}\Bigr)\Bigr],
  \qquad x > 0 .
\]
Substituting $\ln x = \mu + \sigma z$ with $Z$ standard normal, for
every $n \in \mathbb{N}_0$:
\[
\begin{aligned}
  \int_0^\infty x^n f(x)
  \sin\!\Bigl(\tfrac{2\pi(\ln x-\mu)}{\sigma^2}\Bigr)\,dx
  &= e^{n\mu}\,\mathbb{E}\Bigl[e^{n\sigma Z}
    \sin\!\Bigl(\tfrac{2\pi Z}{\sigma}\Bigr)\Bigr]\\
  &= e^{n\mu}\, e^{(n^2\sigma^2 - 4\pi^2/\sigma^2)/2}\,
    \sin(2\pi n) = 0,
\end{aligned}
\]
using $\mathbb{E}[e^{(\alpha+i\beta')Z}] = e^{(\alpha+i\beta')^2/2}$,
whose imaginary part is $e^{(\alpha^2-\beta'^2)/2}\allowbreak\sin(\alpha\beta')$,
with $\alpha = n\sigma$, $\beta' = 2\pi/\sigma$, $\alpha\beta' =
2\pi n$.  The case $n = 0$ shows each $f_c$ is a probability density
(and $f_c \ge 0$ since $|c| \le 1$); the cases $n \ge 1$ show all
$f_c$ share the log-normal moments $\exp(n\mu + n^2\sigma^2/2)$.
Hence uncountably many distinct laws realize the moment sequence.
\end{proof}

\begin{remark}
Indeterminacy cannot be inferred from the convergence of the
Carleman sum: Carleman's condition is sufficient for determinacy but
not necessary, so its failure proves nothing.  The explicit family
above is the classical argument.
\end{remark}

\begin{theorem}[Stability]\label{thm:stability}
Let $\{W_n\}$ be an i.i.d.\ multiplicative cascade with
log-infinitely-divisible generator, all moments finite, and
nontrivial intermittency.  Suppose there exist $\beta \in (0,1)$,
$d^* \in \mathbb{R}$ and $\varepsilon > 0$ such that
\[
  \bigl|\delta_{p+k} - (1-\beta)\,d^* - \beta\,\delta_p\bigr|
  < \varepsilon
  \qquad\text{for all } p \in k\mathbb{N}_0 .
\]
Then $\sigma^2 = 0$, $\supp\nu \subseteq (-\infty,0]$,
$\int(|x|\wedge1)\,d\nu < \infty$, the limit $\delta_\infty :=
\lim_{p\to\infty}\delta_p$ exists finitely, and
$(1-\beta)\,|d^* - \delta_\infty| \le \varepsilon$.  Moreover, with
$u = e^{kx}$, $\eta = (1-u)\,d\tilde\nu$, and
$A = (\delta_0 - \delta_\infty)\ln r$:

\textup{(i)} if $d^* = \delta_\infty$ (the true limit is known),
\[
  W_1\!\left(\frac{\eta}{\|\eta\|},\;\delta_\beta\right)
  \;\leq\;
  \left(\frac{(1+\beta)\,|\ln r|}{|A|}\right)^{\!1/2}
  \!\sqrt{\varepsilon}\,;
\]

\textup{(ii)} in general (fitted $d^*$),
\[
  W_1\!\left(\frac{\eta}{\|\eta\|},\;\delta_\beta\right)
  \;\leq\;
  \left(\frac{2\,|\ln r|}{|A|}\right)^{\!1/2}
  \!\sqrt{\varepsilon}\,.
\]
Both constants are sharp in their respective settings.  In
particular, the cascade multiplier distribution converges to
log-Poisson as $\varepsilon \to 0$, at the sharp rate
$\Theta(\sqrt\varepsilon)$
(Theorems~\ref{thm:propagation} and~\ref{thm:lower}).
\end{theorem}

\begin{proof}
\textit{Step~0 (reduction and existence of the limit).}  From the
hypothesis, $|\delta_{(m+1)k}| \le \beta|\delta_{mk}| +
(1-\beta)|d^*| + \varepsilon$, so the lattice sequence
$\{\delta_{mk}\}$ is bounded:
$\limsup_m |\delta_{mk}| \le |d^*| + \varepsilon/(1-\beta)$.  But by
Steps~2, 3 and~3\textonehalf{} of the proof of
Theorem~\ref{thm:classification}---none of which used the exact
form of A1, only the finiteness of $\liminf$ of
$\{\delta_p\}$---each of the conditions $\sigma^2 > 0$,
$\nu((0,\infty)) > 0$, $\int_{|x|\le1}|x|\,d\nu = \infty$ forces
$\delta_{mk} \to -\infty$, a contradiction.  Hence $\sigma^2 = 0$,
$\supp\nu \subseteq (-\infty,0]$, $\int(|x|\wedge1)\,d\nu < \infty$;
the split form of $\phi$ is valid, dominated convergence gives
$\phi(p) \to c_0$, and $\delta_\infty = c_0/\ln r$ exists finitely.
Letting $m \to \infty$ in the hypothesis,
\[
  \bigl|\delta_\infty - (1-\beta)d^* - \beta\delta_\infty\bigr|
  \le \varepsilon
  \quad\Longrightarrow\quad
  (1-\beta)\,|d^* - \delta_\infty| \le \varepsilon .
\]

\textit{Step~1 (exact moment identities and residuals).}  As in
Step~4 of Theorem~\ref{thm:classification} (now with no
approximation in the identity itself),
\[
  \mu_m := \int_{(0,1)} u^m\,d\eta
  \;=\; \bigl(\delta_{mk} - \delta_\infty\bigr)\,|\ln r|
  \;\ge\; 0
  \qquad\text{for all } m \ge 0,
\]
in particular $\|\eta\| = \mu_0 = (\delta_0 - \delta_\infty)|\ln r| =
|A|$ \emph{exactly}.  Define the signed residuals
\[
  \epsilon_m := \delta_{(m+1)k} - (1-\beta)\,\delta_\infty
  - \beta\,\delta_{mk}
  \;=\; \frac{\mu_{m+1} - \beta\,\mu_m}{|\ln r|},
\]
and let $h_m$ denote the hypothesis residuals (with $d^*$ in place
of $\delta_\infty$), $|h_m| < \varepsilon$.  Setting
$t := (1-\beta)(d^* - \delta_\infty)$, $|t| \le \varepsilon$ by
Step~0, one has $\epsilon_m = h_m + t$.  Under reading~(i), $t = 0$
and $|\epsilon_m| < \varepsilon$.

\textit{Step~2 (variance identity).}  Telescoping,
\[
\begin{aligned}
  \int_{(0,1)} (u-\beta)^2\,d\eta
  &= \mu_2 - 2\beta\mu_1 + \beta^2\mu_0
  = (\mu_2 - \beta\mu_1) - \beta(\mu_1 - \beta\mu_0)\\
  &= |\ln r|\,\bigl(\epsilon_1 - \beta\,\epsilon_0\bigr).
\end{aligned}
\]
Under reading~(i): $|\epsilon_1 - \beta\epsilon_0| < (1+\beta)
\varepsilon$.  Under reading~(ii):
$\epsilon_1 - \beta\epsilon_0 = (h_1 - \beta h_0) + (1-\beta)t$, so
$|\epsilon_1 - \beta\epsilon_0| < (1+\beta)\varepsilon +
(1-\beta)\varepsilon = 2\varepsilon$.  Hence
\[
  0 \;\le\; \int_{(0,1)}(u-\beta)^2\,d\eta
  \;\le\;
  \begin{cases}
    (1+\beta)\,|\ln r|\,\varepsilon & \text{(i)},\\[2pt]
    2\,|\ln r|\,\varepsilon & \text{(ii)}.
  \end{cases}
\]

\textit{Step~3 (Wasserstein bound).}  For a Dirac target the
$W_1$ distance is the first absolute moment:
$W_1(\eta/\|\eta\|, \delta_\beta) = \|\eta\|^{-1}\int|u-\beta|
\,d\eta$.  By Cauchy--Schwarz and $\|\eta\| = |A|$,
\[
\begin{aligned}
  W_1\!\left(\frac{\eta}{\|\eta\|},\,\delta_\beta\right)
  &\;\le\; \Bigl(\frac{1}{|A|}\int (u-\beta)^2\,d\eta\Bigr)^{1/2}\\
  &\;\le\;
  \begin{cases}
    \bigl((1+\beta)\,|\ln r|/|A|\bigr)^{1/2}\sqrt{\varepsilon}
      & \text{(i)},\\[2pt]
    \bigl(2\,|\ln r|/|A|\bigr)^{1/2}\sqrt{\varepsilon}
      & \text{(ii)}.
  \end{cases}
\end{aligned} \qedhere
\]
\end{proof}

\begin{remark}[Sharpness; a warning about per-moment transfer]
\textup{(a)} The constant in reading~(i) is attained in the limit by
the two-atom family $\eta = c_0\delta_{u_0} + c_1\delta_v$ with
$u_0 \downarrow 0$, $v \in (\beta,1)$, $c_1 v(v-\beta) =
\varepsilon|\ln r|$ and $c_0 = (\varepsilon|\ln r|/\beta)(1 + 1/v)$:
then $\epsilon_0 \to -\varepsilon$, $\epsilon_1 = +\varepsilon$,
$|\epsilon_m| = \varepsilon v^{m-1} \le \varepsilon$ for $m \ge 2$,
and $\int(u-\beta)^2 d\eta \to (1+\beta)|\ln r|\varepsilon$.

\textup{(b)} The constant in reading~(ii) is likewise attained:
choose moment residuals $(\epsilon_0, \epsilon_1) \to (0,
2\varepsilon)$ realized by two atoms (one near $0$, one in
$(\beta,1)$) and the offset $t \to \varepsilon$; all hypothesis
residuals $h_m = \epsilon_m - t$ then stay below $\varepsilon$ in
absolute value.  For $\beta < \sqrt2 - 1$ one has $2 >
(1+\beta)^2$: the distinction between the two readings is
quantitatively real, and a constant valid in reading~(i)---even the
non-sharp $(1+\beta)^2$---can \emph{fail} outright in reading~(ii).

\textup{(c)} A tempting route passes through the per-moment estimate
$|\mu_m - |A|\beta^m| \le |\ln r|\,\varepsilon$ for all $m$.  That
estimate is false in general: recurrence errors accumulate to
$|\mu_m - |A|\beta^m| \le |\ln r|\,\varepsilon\,
\tfrac{1-\beta^m}{1-\beta}$, and the bound is attained in the limit
by $\eta = (|A|-c)\delta_\beta + c\,\delta_v$ with $c(v-\beta) =
\varepsilon|\ln r|$, $v \to 1$, for which
$\sup_m |\mu_m - |A|\beta^m| / (|\ln r|\varepsilon) \to
1/(1-\beta)$.  It is also unnecessary: only $\epsilon_0$ and
$\epsilon_1$ enter the variance identity.
\end{remark}

\section{Propagation: the sharp rate}
\label{sec:propagation}

We now transfer the bound of Theorem~\ref{thm:stability} from the
L\'evy measure to the multiplier distribution, at a rate that is
exact: $O(\sqrt\varepsilon)$ with no further hypotheses
(Theorem~\ref{thm:propagation}), and no better
(Theorem~\ref{thm:lower}).

Throughout this section the hypotheses are those of
Theorem~\ref{thm:stability}, reading~(i) (for reading~(ii) replace
$(1+\beta)$ by $2$ in every constant), so that Step~0 there gives
$\sigma^2 = 0$, $\supp\nu \subseteq (-\infty,0]$,
$\int(|x|\wedge1)\,d\nu < \infty$, and Steps~2--3 give, with
$V_\varepsilon := (1+\beta)|\ln r|\,\varepsilon$ and
$S := \bigl((1+\beta)\,|A|\,|\ln r|\bigr)^{1/2}$,
\begin{equation}\label{eq:VF}
  \int_{(0,1)}(u-\beta)^2\,d\eta \;\le\; V_\varepsilon,
  \qquad
  \int_{(0,1)}|u-\beta|\,d\eta \;\le\; S\sqrt{\varepsilon} .
\end{equation}
The L\'evy measure may have \emph{infinite total mass},
accumulating only at $u = 1$ where the tilt $(1-u)$ vanishes; the
multiplier is the a.s.-convergent product over the Poisson point
process $\{U_i\}$ of intensity $\tilde\nu$,
\[
  W = e^{a_\varepsilon}\prod_i U_i^{1/k},
  \quad
  \sum_i\bigl(1 - U_i^{1/k}\bigr) \le \tfrac2k\sum_i(1-U_i)
  \ \text{ of finite mean } \tfrac2k\|\eta\|
\]
(Campbell's formula \cite{Kin93}).  Under $\mathbb{E}[W]=1$ the
drift is $a_\varepsilon = \int(1-u^{1/k})\,d\tilde\nu < \infty$.
The comparison target $W_0$ is the log-Poisson multiplier with jump
factor $\beta^{1/k}$, rate $\lambda = |A|/(1-\beta)$, and drift
$a_0$ fixed by the same normalization.

\begin{theorem}[Unconditional propagation]\label{thm:propagation}
Under the hypotheses of Theorem~\textup{\ref{thm:stability}}
alone---no finite-activity or minimum-jump assumption---there exist
$\varepsilon_0 > 0$ and an explicit constant $K_\infty =
K_\infty(\beta, C, r, k)$ such that for all $\varepsilon \le
\varepsilon_0$,
\[
  W_1\bigl(\mathrm{law}(W),\ \mathrm{law}(W_0)\bigr)
  \;\le\; K_\infty\,\sqrt{\varepsilon}.
\]
One admissible (not optimized) choice is
\[
\begin{aligned}
  K_\infty &= 4\,e^{a_0+1}\Bigl(\frac{L_k}{1-\beta}
  + \frac{1}{(1-\beta)^2}\Bigr)
  \sqrt{(1+\beta)\,|A|\,|\ln r|}\,,\\
  L_k &= \tfrac1k\bigl(\tfrac\beta2\bigr)^{(1-k)/k},
\end{aligned}
\]
absorbing $O(\varepsilon)$ terms via $\varepsilon \le
\sqrt\varepsilon$.
\end{theorem}

\begin{proof}
Fix the $\varepsilon$-independent split height
\[
  h_0 := \tfrac{1-\beta}{2},
  \qquad
  u_0 := 1 - h_0 = \tfrac{1+\beta}{2},
\]
and call a jump \emph{small} if $u \in (u_0, 1)$, \emph{macroscopic}
if $u \in (0, u_0]$.

\textit{Step~1 (small jumps: individually cheap, collectively
$O(\varepsilon)$).}  Every $u \in (u_0,1)$ lies at distance
$> h_0$ from $\beta$, so by Chebyshev's inequality
against~\eqref{eq:VF},
\[
  \eta\bigl((u_0,1)\bigr) \;\le\; \frac{V_\varepsilon}{h_0^{\,2}}
  \;=\; \frac{4(1+\beta)|\ln r|}{(1-\beta)^2}\,\varepsilon .
\]
Leave the small-jump points of $W$ \emph{unpaired} and cost them
multiplicatively: for factors in $(0,1]$, telescoping gives
$|\prod_i s_i - 1| \le \sum_i(1-s_i)$ (valid for infinite products
by monotone limits), and $1 - u^{1/k} \le \tfrac2k(1-u)$ for $u \ge
\tfrac12$ (derivative bound; $u_0 \ge \tfrac12$).  By Campbell's
formula,
\[
\begin{aligned}
  \mathbb{E}\Bigl|\prod_{\text{small}} U_i^{1/k} - 1\Bigr|
  \;&\le\; \frac2k\int_{(u_0,1)}(1-u)\,d\tilde\nu
  \;=\; \frac2k\,\eta\bigl((u_0,1)\bigr)\\
  \;&\le\; \frac{8(1+\beta)|\ln r|}{k(1-\beta)^2}\,\varepsilon .
\end{aligned}
\]
\emph{This is the step that controls infinite activity: a near-1
jump's cost is its multiplicative deviation $\asymp(1-u)$---already
$\eta$-weighted---not the count~$1$.}

\textit{Step~2 (the macroscopic part is automatically
finite-activity, with a fixed de-tilting constant).}  On $(0,u_0]$:
$\tfrac1{1-u} \le \tfrac1{h_0} = \tfrac2{1-\beta}$, so
$\lambda_{\mathrm{mac}} := \tilde\nu((0,u_0]) \le
\tfrac{2|A|}{1-\beta} < \infty$.  Comparing rates against
$\lambda = |A|/(1-\beta) = \int\tfrac{d\eta}{1-\beta}$, and using
$\bigl|\tfrac1{1-u} - \tfrac1{1-\beta}\bigr| =
\tfrac{|u-\beta|}{(1-u)(1-\beta)} \le
\tfrac{2|u-\beta|}{(1-\beta)^2}$ on $(0,u_0]$,
\[
  \bigl|\lambda_{\mathrm{mac}} - \lambda\bigr|
  \;\le\; \frac{2}{(1-\beta)^2}\,S\sqrt{\varepsilon}
  \;+\; \frac{\eta((u_0,1))}{1-\beta},
\]
the second term being $O(\varepsilon)$ by Step~1.

\textit{Step~3 (macroscopic jump cost).}  The map $u \mapsto
u^{1/k}$ is $L_k$-Lipschitz on $[\beta/2,1]$ and bounded by~1
below $\beta/2$; on $(0,\beta/2)$, $\tfrac1{1-u} \le
\tfrac1{1-\beta/2} \le 2$, so
$\tilde\nu((0,\beta/2)) \le 2\eta((0,\beta/2)) \le
2V_\varepsilon(\beta/2)^{-2}$ by Chebyshev.  Hence
\[
  \int_{(0,u_0]}\bigl|u^{1/k} - \beta^{1/k}\bigr|\,d\tilde\nu
  \;\le\; \frac{2L_k}{1-\beta}\,S\sqrt{\varepsilon}
  \;+\; \frac{8\,V_\varepsilon}{\beta^2}.
\]

\textit{Step~4 (assembly).}  Couple the macroscopic Poisson count
with $W_0$'s count by thinning (shared count
$\mathrm{Poisson}(\lambda\wedge\lambda_{\mathrm{mac}})$, excess
independent; \cite{BHJ92}, Theorem~10.A), shared jump pairs
optimally in the multiplier coordinate against the constant target
$\beta^{1/k}$, and leave the small jumps unpaired.  Writing $W =
e^{a_\varepsilon}P_{\mathrm{mac}}P_{\mathrm{small}}$ and $W_0 =
e^{a_0}P_0$ with all products in $[0,1]$, the telescoping inequality
and Wald's identity give
\[
\begin{aligned}
  \mathbb{E}\bigl|P_{\mathrm{mac}}P_{\mathrm{small}} - P_0\bigr|
  \;\le\;
  &\underbrace{\int_{(0,u_0]}\!\bigl|u^{1/k}-\beta^{1/k}\bigr|
  d\tilde\nu}_{\text{Step 3}}
  \;+\;
  \underbrace{\bigl|\lambda_{\mathrm{mac}} -
  \lambda\bigr|}_{\text{Step 2}}\\
  &+\;
  \underbrace{\mathbb{E}\bigl|P_{\mathrm{small}} -
  1\bigr|}_{\text{Step 1}},
\end{aligned}
\]
each excess jump on either side changing a product by at most~1.
The drift difference obeys the same bound: under $\mathbb{E}[W]=1$,
$a_\varepsilon - a_0 = \int(1-u^{1/k})\,d\tilde\nu -
\lambda(1-\beta^{1/k})$ decomposes into the same three pieces.  Both
multipliers are bounded by $e^{a_\cdot}$
(Lemma~\ref{lem:determinacy}(i)), so for $\varepsilon \le
\varepsilon_0$ with $|a_\varepsilon - a_0| \le 1$,
\[
\begin{aligned}
  W_1 \le \mathbb{E}|W - W_0|
  &\le e^{a_0+1}\Bigl(\mathbb{E}|P_{\mathrm{mac}}P_{\mathrm{small}}
  - P_0| + |a_\varepsilon - a_0|\Bigr)\\
  &\le 2e^{a_0+1}\bigl[\text{Step 1} + \text{Step 2} +
  \text{Step 3}\bigr].
\end{aligned}
\]
Substituting the three displays and absorbing every
$O(\varepsilon)$ term via $\varepsilon \le \sqrt\varepsilon$ yields
the stated $K_\infty$.
\end{proof}

\begin{remark}[Why no hypotheses are needed]
A coupling built in log-space would require finite activity and a
minimum jump size---a Poisson count must be finite, and a de-tilting
factor must be bounded---and an $\varepsilon$-dependent shell
decomposition for the small jumps would surrender a logarithmic
factor, $O(\sqrt\varepsilon\log(1/\varepsilon))$.  In the
multiplicative coupling both holes close themselves: infinite
activity can only accumulate at $u = 1$, where multiplicative cost
vanishes at exactly the rate $\eta$ measures, and the de-tilting
factor appears only on $(0,u_0]$, where it is the fixed constant
$2/(1-\beta)$.  The flatness of the exponential map at $-\infty$,
fatal to couplings of jump distributions in $x$-space, is the
resource here.
\end{remark}

\begin{theorem}[Lower bound: the rate $\sqrt\varepsilon$ is
exact]\label{thm:lower}
Fix $(\beta, C, r)$, $k = 1$, and conservation $\mathbb{E}[W]=1$.
For $d \in (0, d_0]$, $d_0 = d_0(\beta)$ small, let $W_{(d)}$ be the
compound-Poisson multiplier with tilted measure
\[
  \eta_d = \tfrac{|A|}{2}\bigl(\delta_{\beta-d} +
  \delta_{\beta+d}\bigr)
\]
(jump atoms $\beta\pm d$ with L\'evy masses
$\tfrac{|A|/2}{1-(\beta\pm d)}$, drift by conservation).  Then:

\textup{(i)} its A1 residual satisfies
$\dfrac{|A|}{|\ln r|}\,d^2 \;\le\; \varepsilon(d) \;\le\;
c_2(\beta)\,\dfrac{|A|}{|\ln r|}\,d^2$, where $c_2(\beta) =
\bigl[e\min_{\rho\in[\beta,(1+\beta)/2]}\rho|\ln\rho|\bigr]^{-1}
\vee 1$;

\textup{(ii)} with $\lambda_d = \|\tilde\nu_d\| =
\lambda/(1 - d^2(1-\beta)^{-2})$ and $a = |A|$,
\[
  W_1\bigl(\mathrm{law}\,W_{(d)},\ \mathrm{law}\,W_0\bigr)
  \;\ge\; \lambda_d\,e^{-\lambda_d}\,e^{a}\,d
  \;\ge\; c_1(\beta, C, r)\,\sqrt{\varepsilon(d)}\,.
\]
Hence no propagation bound of order $o(\sqrt\varepsilon)$ is
possible: combined with Theorem~\ref{thm:propagation}, the exact
rate is $\Theta(\sqrt\varepsilon)$.
\end{theorem}

\begin{proof}
\textit{(i)}  The moments are $\mu_m =
\tfrac{|A|}{2}[(\beta-d)^m + (\beta+d)^m]$, so the signed residuals
are
\[
  \epsilon_m = \frac{\mu_{m+1} - \beta\mu_m}{|\ln r|}
  = \frac{|A|\,d}{2\,|\ln r|}
  \bigl[(\beta+d)^m - (\beta-d)^m\bigr] \;\ge 0,
\]
with $\epsilon_0 = 0$ and $\epsilon_1 = |A|d^2/|\ln r|$ exactly,
giving the lower bound on $\varepsilon(d) = \sup_m \epsilon_m$.  For
the upper bound, the mean value theorem gives $(\beta+d)^m -
(\beta-d)^m \le 2dm(\beta+d)^{m-1}$, and $\sup_{x\ge0} x\rho^{x-1} =
(e\rho|\ln\rho|)^{-1}$ for $\rho = \beta+d \le (1+\beta)/2$.

\textit{(ii)}  For $k=1$ the conservation drift is
$a_{(d)} = \int(1-u)\,d\tilde\nu_d = \|\eta_d\| = |A|$, identical to
$a_0 = \lambda(1-\beta) = |A|$: the family is \emph{drift-rigid}.
The support of $\mathrm{law}(W_0)$ is the geometric set $G =
\{e^{a}\beta^j : j \ge 0\}$.  Take the 1-Lipschitz test function
\[
  f(w) := \min\Bigl(\dist\bigl(w,\,G\bigr),\
  \tfrac{e^a\beta(1-\beta)}{2}\Bigr) \;\ge\; 0,
\]
which vanishes on $G$, so $\mathbb{E}f(W_0) = 0$.  On the one-jump
event of $W_{(d)}$ (probability $\lambda_d e^{-\lambda_d}$),
$W_{(d)} = e^{a}(\beta\pm d)$, whose distance to the nearest point
of $G$ is exactly $e^a d$ for $d < \beta(1-\beta)/2$ (the
neighbors $e^a$ and $e^a\beta^2$ are farther, and $e^ad$ is below
the cap); all other events contribute $\ge 0$.  By
Kantorovich--Rubinstein duality \cite{Vil09},
\[
  W_1 \;\ge\; \mathbb{E}f(W_{(d)}) - \mathbb{E}f(W_0)
  \;\ge\; \lambda_d e^{-\lambda_d}\,e^{a}\,d,
\]
and substituting $d \ge \bigl(|\ln r|\,\varepsilon(d) / (c_2(\beta)
|A|)\bigr)^{1/2}$ from~(i) gives the $c_1\sqrt{\varepsilon}$ form.
(The same construction works for $k > 1$ on the lattice
$\{e^{a}\beta^{j/k}\}$; $k = 1$ is stated for cleanliness.)
\end{proof}

\begin{remark}[The stability theory is elementary and now
complete]
The change of variables $u = e^{kx}$ maps the L\'evy measure to the
compact interval $[0,1]$, where a second-moment test against the
candidate atom decides everything: the classification is the
$\varepsilon = 0$ case of the variance identity in
Theorem~\ref{thm:stability}, Step~2; the stability constant is read
off two residuals and is sharp; and the propagation to the
multiplier law is a telescoped multiplicative coupling with no tail
estimates (boundedness, Lemma~\ref{lem:determinacy}(i)), sharp in
rate by Theorem~\ref{thm:lower}.  Every quantitative statement in
the package---variance constant $(1+\beta)$, Wasserstein constant,
propagation rate $\Theta(\sqrt\varepsilon)$---is attained by an
explicit family.  In particular the log-Poisson class is an open
set, with exactly computed modulus, in the space of cascade
multiplier distributions metrized by A1 residuals.
\end{remark}

\section{Beyond independence: stationary and Markov multipliers}
\label{sec:beyond}

The i.i.d.\ assumption enters the preceding sections through the
identity $\mathbb{E}[(W_1\cdots W_n)^p] = (\mathbb{E}[W^p])^n$,
which converts scaling data into one-step moment data.  This section
determines exactly what survives without it.  Let $(W_n)$ be
\emph{stationary ergodic}, $X_n = \ln W_n$, $S_n = X_1 + \cdots +
X_n$, and define exponents asymptotically:

\smallskip
\noindent\textbf{Standing assumptions (S).}  For each lattice $p$:
$m_n(p) := \mathbb{E}[e^{pS_n}] < \infty$ for all $n$, and
\[
  \Lambda(p) := \lim_{n\to\infty}\tfrac1n\ln m_n(p)
  \quad\text{exists and is finite};
  \qquad \zeta_p := \Lambda(p)/\ln r .
\]
A1 is imposed on $\delta_p = \zeta_{p+k}-\zeta_p$ as before.  Write
$\Lambda_{\mathrm{LP}}(p) = ap + \lambda(e^{bp}-1)$ for the
log-Poisson limiting cumulant function with the parameter dictionary
of Theorem~\ref{thm:characterization}.

\begin{theorem}[Asymptotic-statistics classification]\label{thm:cgflevel}
Under \textup{(S)}, A1 holds on the lattice if and only if
$\Lambda = \Lambda_{\mathrm{LP}}$ on the lattice.  Consequently,
under A1 every observable computed from lattice exponents---%
structure-function exponents, moment ratios, and (under the
regularity below) the large-deviation rate function and multifractal
spectrum---coincides exactly with that of the i.i.d.\ log-Poisson
cascade with parameters $(\beta, \gamma, C)$.
\end{theorem}

\begin{proof}
Lemma~\ref{lem:exponent} is purely algebraic, so A1 gives
$\Lambda(km) = a\,km + \lambda(e^{b\,km}-1)$ for all $m$; the
converse is the computation of Proposition~\ref{prop:converse}.
\end{proof}

\begin{remark}
If $\Lambda$ exists, is finite and differentiable on a neighborhood
of $[0,\infty)$, the G\"artner--Ellis theorem \cite{DZ98} yields a
large-deviation principle for $S_n/n$ with rate $I(h) =
\sup_p[ph - \Lambda(p)]$ on the exposed range---the Legendre
structure of Theorem~\ref{thm:characterization}(iii).  Off-lattice,
$\Lambda$ is pinned between consecutive lattice values by convexity.
\end{remark}

Theorem~\ref{thm:cgflevel} is deliberately easy; the substantive
question is whether A1 still determines the multiplier \emph{law}.
It does not:

\begin{theorem}[Interleaved cascade: the law is not
determined]\label{thm:interleaved}
There exists a stationary ergodic multiplier sequence
$(W_n)$---realizable as a function of a stationary, irreducible,
positive-recurrent countable-state Markov chain---such that:

\textup{(i)} $m_n(p) = r^{n\zeta^{\mathrm{LP}}_p}$ \emph{exactly} for
every even $n$ and every real $p \ge 0$; hence \textup{(S)} holds,
$\Lambda = \Lambda_{\mathrm{LP}}$ on all of $[0,\infty)$, and A1
holds exactly, in its strongest (real-$p$) form;

\textup{(ii)} the one-step marginal of $\ln W_1$ is \emph{not}
log-Poisson;

\textup{(iii)} the sequence is not i.i.d., and is not equal in law
to any i.i.d.\ cascade.

Consequently the law-level conclusion of
Theorem~\ref{thm:characterization} does not extend beyond
independence, by any proof.
\end{theorem}

\begin{proof}
\textit{Construction.}  With $(a, b, \lambda)$ as in
Theorem~\ref{thm:characterization} and $\lambda' := 2\lambda$, let
$A_1, A_2, \ldots$ be i.i.d.\ $\mathrm{Poisson}(\lambda')$, define
\[
  Y_{2j-1} = a + bA_j, \qquad Y_{2j} = a \qquad (j \ge 1),
\]
draw a phase $\theta \sim \mathrm{Unif}\{0,1\}$ independent of
everything, and set $X_n := Y_{n+\theta}$, $W_n := e^{X_n}$: jump
slots of doubled intensity alternate with deterministic slots.

\textit{Stationarity and Markov realization.}  The pair $Z_n :=
((n+\theta)\bmod 2,\, X_n)$ is a Markov chain on
$\{0,1\}\times(a+b\mathbb{N}_0)$: from phase-1 states the next value
is $a + bA$ with fresh $A \sim \mathrm{Poisson}(\lambda')$ and the
phase flips; from phase-0 states the next value is $a$ and the phase
flips.  The chain is irreducible on its reachable set and positive
recurrent; the phase-uniform stationary law makes $(Z_n)$
stationary and $W_n$ a function of it.

\textit{Ergodicity.}  Let $P_0, P_1$ be the path laws given $\theta
= 0,1$, so the law is $\tfrac12(P_0+P_1)$ with $P_1 = P_0\circ
T^{-1}$ ($T$ = shift).  If $E$ is $T$-invariant, it is
$T^2$-invariant, and under $P_0$ the double shift is ergodic (the
blocks $(Y_{2j-1}, Y_{2j})$ are i.i.d.), so $P_0(E) \in \{0,1\}$;
and $P_1(E) = P_0(T^{-1}E) = P_0(E)$.  Hence $\mathbb{P}(E) \in
\{0,1\}$.

\textit{Exact exponents.}  For even $n = 2m$ the window
$\{1,\ldots,2m\}$ contains exactly $m$ jump slots under either
phase, carrying $m$ distinct i.i.d.\ $A_j$'s; hence, exactly, for
every real $p$,
\[
\begin{aligned}
  m_{2m}(p)
  &= e^{2map}\exp\bigl[m\lambda'(e^{bp}-1)\bigr]
  = \exp\bigl[2m\bigl(ap + \lambda(e^{bp}-1)\bigr)\bigr]\\
  &= \exp\bigl[2m\,\Lambda_{\mathrm{LP}}(p)\bigr].
\end{aligned}
\]
For odd $n$ the window covers $m$ or $m+1$ jump slots depending on
the phase, and $\tfrac1n\ln m_n(p) \to \Lambda_{\mathrm{LP}}(p)$.
Theorem~\ref{thm:cgflevel} then gives A1 exactly.

\textit{Non-log-Poisson marginal.}  $\mathbb{P}(X_1 = a) =
\tfrac12(1 + e^{-2\lambda})$, whereas the log-Poisson($\lambda$)
marginal has $\mathbb{P}(X = a) = e^{-\lambda}$; for She--L\'ev\^eque
dissipation parameters ($\lambda = 2\ln2$): $0.531$ versus $0.250$.
The marginal is the mixture $\tfrac12\delta_a +
\tfrac12\,\mathrm{law}(a + b\,\mathrm{Poisson}(2\lambda))$.

\textit{Non-i.i.d.}  Given $X_n \ne a$ the next multiplier is
deterministic: $\mathbb{P}(X_{n+1} = a \mid X_n \ne a) = 1 \ne
\mathbb{P}(X_{n+1} = a)$.
\end{proof}

\begin{remark}
The construction generalizes freely (blocks of length $L$, arbitrary
allocation of the total jump intensity across slots,
Markov-modulated loads): A1 is compatible with an
infinite-dimensional family of mutually singular stationary ergodic
processes, all sharing the log-Poisson asymptotics---exactly as
Theorem~\ref{thm:cgflevel} says they must.
\end{remark}

Two rigidity results delimit the boundary of
Theorem~\ref{thm:interleaved}.

\begin{corollary}[Exchangeable rigidity]\label{cor:exch}
Let $(W_n)$ be exchangeable with $m_n(p) = r^{n\zeta_p}$ holding
exactly for $n \in \{1,2\}$ and all lattice $p$, with $\zeta$
satisfying A1.  Then $(W_n)$ is i.i.d.\ log-Poisson with parameters
$(\beta, \gamma, C)$.
\end{corollary}

\begin{proof}
By de Finetti's theorem \cite{Kal02}, $(W_n)$ is conditionally
i.i.d.\ given a random directing measure; let $M_p :=
\mathbb{E}[W_1^p \mid \text{directing measure}]$.  Conditional
independence gives $m_1(p) = \mathbb{E}[M_p]$ and $m_2(p) =
\mathbb{E}[M_p^2]$, so exactness at $n = 1,2$ reads
$\mathbb{E}[M_p] = r^{\zeta_p}$, $\mathbb{E}[M_p^2] =
r^{2\zeta_p}$, whence $\Var(M_p) = 0$: $M_p = r^{\zeta_p}$ a.s., for
every lattice $p$ simultaneously.  Almost every directing measure
therefore has exactly the A1 lattice moments, hence equals
\emph{the} log-Poisson law by Lemma~\ref{lem:determinacy} and
Theorem~\ref{thm:characterization}(ii); the mixture is degenerate.
\end{proof}

\begin{theorem}[Finite-state impossibility]\label{thm:finitestate}
Let $\xi$ be an irreducible finite-state Markov chain, stationary,
and $W_n = e^{f(\xi_n)}$ with $f$ non-constant.  Then the cascade
cannot satisfy A1 in its real-$p$ form with nontrivial
intermittency: there are no parameters $(\beta, \gamma, C)$ with
$C > 0$ and $\zeta_p = \gamma p + C(1-\beta^{p/k})$ for all real
$p \ge 0$.
\end{theorem}

\begin{proof}
For finite irreducible chains, $\Lambda(p) = \ln\rho(M(p))$ with
$M(p)_{xy} = P_{xy}e^{pf(y)}$ and $\rho$ the Perron root---a simple
eigenvalue for every real $p$, hence real-analytic on $\mathbb{R}$.
If the closed form held on $[0,\infty)$, then $\Lambda(p) = ap +
\lambda(e^{bp}-1)$ there with $\lambda = -C\ln r > 0$; both sides
are real-analytic on $\mathbb{R}$ and agree on an interval, hence
agree everywhere (identity theorem).  Now let $p \to -\infty$.
Since $f_{\min} = \min_s f(s) > -\infty$, every row sum of $M(p)$
is at most $e^{pf_{\min}}$ for $p \le 0$, so $\Lambda(p) \le
pf_{\min}$: a linear upper bound.  But with $b < 0$,
\[
  \Lambda_{\mathrm{LP}}(p) - pf_{\min}
  \;\ge\; \lambda e^{|b||p|} - \lambda + p(a - f_{\min})
  \;\longrightarrow\; +\infty :
\]
contradiction.
\end{proof}

\begin{remark}[The boundary, and what A1 really is]
The mechanism of Theorem~\ref{thm:finitestate} is that finite
alphabets make $\ln W$ bounded below, while the log-Poisson
generator is intrinsically unbounded below ($\ln W \in a +
b\mathbb{N}_0$): A1 forces multipliers with arbitrarily severe
attenuation events, and the counterexample of
Theorem~\ref{thm:interleaved} necessarily has unbounded-below
$\ln W$.  Assembled, this section says: A1 is an
\emph{asymptotic-statistics} axiom.  It pins the limiting cumulant
function, the spectrum, and the large deviations to the log-Poisson
cascade (Theorem~\ref{thm:cgflevel}); it cannot pin the per-step law
(Theorem~\ref{thm:interleaved}); and the i.i.d.\ log-Poisson cascade
is the canonical realization---unique under exchangeability
(Corollary~\ref{cor:exch}), with finite-state Markov realizations
impossible (Theorem~\ref{thm:finitestate}).  Whether
Theorem~\ref{thm:finitestate} persists under lattice-only A1
remains open (the identity-theorem step is unavailable on a discrete
set; see Section~\ref{sec:discussion}).
\end{remark}

\section{Continuous cascades: log-infinitely-divisible multifractal
measures}
\label{sec:mrm}

We now place the theory in the continuous category of Bacry--Muzy
multifractal random measures \cite{BM03,MB02}, which contains the
log-normal multifractal random walk, log-stable measures, and the
Barral--Mandelbrot compound Poisson cascades \cite{BM02} as special
cases.  We use two structural properties of the class as axioms:

\smallskip
\noindent\textbf{(M1) Exact stochastic scale invariance.}  For every
$\sigma \in (0,1)$ and $t \le T$ (the integral scale),
\[
  \bigl(M(\sigma t)\bigr)_t \;\stackrel{d}{=}\;
  \sigma\,e^{\Omega_\sigma}\bigl(M(t)\bigr)_t,
\]
with $\Omega_\sigma$ infinitely divisible, independent of $M$, and
$\mathbb{E}[e^{q\Omega_\sigma}] = \sigma^{-\psi(q)}$, where $\psi$
is the L\'evy exponent of the generator per unit logarithmic scale.

\smallskip
\noindent\textbf{(M2) Conservation.}  $\mathbb{E}[e^{\Omega_\sigma}]
= 1$, i.e.\ $\psi(1) = 0$.

\smallskip
\noindent\textbf{(S$_c$)}  $\psi(q) < \infty$ for all $q \ge 0$ (the
continuous analogue of finite multiplier moments).

\smallskip
From (M1), wherever $\mathbb{E}[M([0,t])^q] < \infty$,
\[
  \mathbb{E}\bigl[M([0,t])^q\bigr] \propto t^{\zeta_q},
  \qquad
  \zeta_q = q - \psi(q),
\]
and moments of the total mass are finite (for $q > 1$) precisely on
the window where $\zeta_q > 1$ \cite{BM03}.  Since $\zeta$ is
concave with $\zeta_1 = 1$, \emph{structure functions see only a
finite window $[0, q^*)$}---a hard information barrier with no
discrete counterpart (there the log-Poisson multiplier is bounded
and all moments exist).  The hierarchical symmetry A1 is imposed on
$\delta_q = \zeta_{q+k} - \zeta_q$ as before.

The dictionary to the discrete theory is one line: with $\phi_c(q)
:= \psi(q+k) - \psi(q)$ one has $\delta_q = k - \phi_c(q)$, and A1
gives, by Lemma~\ref{lem:exponent},
\[
  \phi_c(mk) - \phi_{c,\infty} = A_c\,\beta^m,
  \qquad
  A_c = -(\delta_0 - \delta_\infty) = -C(1-\beta) < 0 :
\]
formally the discrete identities with $\ln r \mapsto -1$, so $|A_c|
= C(1-\beta)$ and the rate $\lambda = -C\ln r$ becomes the intensity
$C$ \emph{per unit log-scale}.  Every $|\ln r|$ in the discrete
constants disappears.

\begin{theorem}[No finite moment window identifies the
class]\label{thm:window}
Fix the She--L\'ev\^eque exponent curve $\zeta^{\mathrm{SL}}_q =
\gamma q + C(1-\beta^{q/k})$ normalized by \textup{(M2)}, and any
finite lattice window $F \subset k\mathbb{N}_0$.  Then there is a
continuum of exactly scale-invariant log-ID multifractal random
measures whose generators are not compound Poisson with a single
atom---in particular are not the compound Poisson cascade---yet
whose exponents satisfy $\zeta_q = \zeta^{\mathrm{SL}}_q$ for every
$q \in F$.  Structure-function data on a finite moment window, even
exact and noise-free, cannot certify the log-Poisson class.
\end{theorem}

\begin{proof}
We give the construction for $k = 1$ and the physically typical
window $F = \{0,1,2\}$ (i.e.\ $2 < q^* \le 3$); larger windows are
identical with more atoms.  In the tilted coordinates $u = e^{x}$,
the CPC generator (Theorem~\ref{thm:cpc}) has $\tilde\Pi = C
\delta_\beta$.  Perturb:
\[
  \tilde\Pi_s = (C - sc)\,\delta_\beta + s\,(w_1\delta_{v_1} +
  w_2\delta_{v_2}),
  \qquad 0 < v_1 < \beta < v_2 < 1,
\]
with the drift re-fixed by (M2) for each $s$, and impose
\[
  \sum_{i=1,2} w_i\,(v_i^m - 1) \;=\; c\,(\beta^m - 1),
  \qquad m = 1, 2 .
\]
The $m=1$ equation makes the (M2)-drifts of $\tilde\Pi_s$ and
$\tilde\Pi_0$ coincide; the $m=2$ equation then matches $\zeta_2$;
$\zeta_0 = 0$ and $\zeta_1 = 1$ are automatic.  At $(\beta, v_1,
v_2) = (2/3,\,0.3,\,0.9)$, $c = 1$, the $2\times2$ system gives
$w_1 = 0.1852$, $w_2 = 2.0370$, both \emph{positive}, so
$\tilde\Pi_s \ge 0$ for all $s \in [0, C/c]$: a one-parameter family
of genuine L\'evy measures, none a single atom for $s > 0$, all
matching $\zeta^{\mathrm{SL}}$ exactly on $F$ (and differing beyond:
at $m = 3$ the perturbed exponent differs by $+0.0143$ at
$s = \tfrac12$).  For general finite $F$, the same ansatz with more
atoms imposes finitely many linear constraints on infinitely many
degrees of freedom, with positivity maintained by anchoring the
negative part on the CPC atom.
\end{proof}

\begin{remark}
Theorem~\ref{thm:window} does not contradict the discrete
classification: there A1 was available at \emph{all} lattice
orders---the divergence steps of Theorem~\ref{thm:classification}
need $q \to \infty$, and on a finite window even $\sigma_0^2 > 0$
survives (a quadratic log-normal exponent interpolates any
three-point window with the correct convexity).  Identifiability
requires constraints of unbounded order, which structure functions
cannot supply.  The scale-invariance factor $\Omega_\sigma$ can:
under \textup{(S$_c$)} it has all exponential moments, at every
$\sigma$, and its statistics (``magnitude'' statistics, in the
language of \cite{DMA01}) are observable at all orders.
\end{remark}

\begin{theorem}[Generator-level classification: A1 selects the
compound Poisson cascade]\label{thm:cpc}
Let $M$ satisfy \textup{(M1)}, \textup{(M2)}, \textup{(S$_c$)}, with
nontrivial intermittency $C > 0$, and define the generator
exponents $\zeta_q = q - \psi(q)$ for all $q \ge 0$.  Then A1 on the
full lattice $k\mathbb{N}_0$ holds if and only if
\[
  \sigma_0^2 = 0, \qquad \Pi = C\,\delta_b, \qquad
  b = \frac{\ln\beta}{k} < 0,
\]
with drift $\tilde a = C(1-\beta^{1/k})$ fixed by \textup{(M2)};
equivalently
\[
  \Omega_\sigma \;\stackrel{d}{=}\;
  \tilde a\,\ln(1/\sigma) \;+\; b\,\mathrm{Poisson}\bigl(C
  \ln(1/\sigma)\bigr)
  \qquad\text{for every } \sigma \in (0,1),
\]
log-Poisson at every scale ratio simultaneously, and $M$ is the
Barral--Mandelbrot compound Poisson cascade \cite{BM02} with fixed
multiplier atom $e^b = \beta^{1/k}$ and Poisson intensity $C$ on the
time--log-scale cone.  No other member of the Bacry--Muzy
class---log-normal, log-stable, or any intermediate---satisfies A1.
\end{theorem}

\begin{proof}
\textit{Reverse:}  with $\Pi = C\delta_b$, $\psi(q) = \tilde aq +
C(e^{bq}-1)$, so $\zeta_q = q - \psi(q) = (1-\tilde a)q +
C(1-\beta^{q/k})$, which satisfies A1 by
Proposition~\ref{prop:converse} with $\gamma = 1 - \tilde a$; (M2)
gives $\tilde a = C(1-\beta^{1/k})$, the continuous conservation
constraint ($\zeta_1 = 1$).

\textit{Forward:}  $\psi$ is a L\'evy--Khintchine exponent, finite
for all $q \ge 0$, and $\phi_c(q) = \psi(q+k)-\psi(q)$ has a finite
limit under A1: exactly the hypothesis configuration of
Theorem~\ref{thm:classification}'s proof under the dictionary
$\ln r \mapsto -1$.  Steps~2, 3, 3\textonehalf{} and~4 of that proof
apply verbatim: a Gaussian component, positive jumps, or
$\int_{|x|\le1}|x|\,d\Pi = \infty$ each force $\delta_q = k -
\phi_c(q) \to -\infty$; then with $\eta = (1-u)\,d\tilde\Pi$,
$u = e^{kx}$, A1 gives $\int u^m\,d\eta = |A_c|\beta^m$ and the
variance identity yields $\eta = |A_c|\delta_\beta$, i.e.\
$\tilde\Pi = \frac{|A_c|}{1-\beta}\delta_\beta = C\delta_\beta$ and
$\Pi = C\delta_b$.  The identification of the compound-Poisson
member of the Bacry--Muzy class with the Barral--Mandelbrot cascade
is \cite{BM03}.
\end{proof}

\begin{corollary}[Stability with native constants]\label{cor:mrmstab}
In the setting of Theorem~\textup{\ref{thm:cpc}}, if A1 holds to
within $\varepsilon$ on the generator lattice---with $d^*$ the true
limit (reading~(i)) or fitted (reading~(ii))---then $\sigma_0^2 =
0$, $\Pi$ is supported on $(-\infty,0]$, and with $\eta =
(1-u)\,d\tilde\Pi$, $\|\eta\| = |A_c| = C(1-\beta)$ exactly:
\[
  W_1\!\Bigl(\frac{\eta}{\|\eta\|},\,\delta_\beta\Bigr)
  \;\le\; \sqrt{\frac{1+\beta}{C(1-\beta)}}\;\sqrt{\varepsilon}
  \quad\textup{(i)},
  \qquad
  \le\; \sqrt{\frac{2}{C(1-\beta)}}\;\sqrt{\varepsilon}
  \quad\textup{(ii)},
\]
both sharp; for She--L\'ev\^eque values $(\beta, C) = (2/3, 2)$:
$1.58\sqrt\varepsilon$ and $1.73\sqrt\varepsilon$.  The propagation
to the law of the per-octave factor $e^{\Omega_{1/2}}$ holds at the
unconditional sharp rate $\Theta(\sqrt\varepsilon)$ of
Theorems~\ref{thm:propagation}--\ref{thm:lower} (whose proofs nowhere
used finite activity---the natural situation here).
\end{corollary}

\begin{proof}
Theorem~\ref{thm:stability}, Steps~0--3, under the dictionary
$\ln r \mapsto -1$, $|A| \mapsto |A_c| = C(1-\beta)$; the sharpness
families transfer symbol-for-symbol, as does the propagation
argument of Section~\ref{sec:propagation}.
\end{proof}

\begin{remark}[What the two theorems mean together]
Theorems~\ref{thm:window} and~\ref{thm:cpc} are two halves of one
statement about observability: structure functions cannot identify
the cascade class even in principle (a finiteness barrier, not a
statistical one), while magnitude statistics classify it completely,
with quantitative stability.  This places a theorem under the
long-standing practical preference for magnitude-cumulant analysis
over high-order structure functions \cite{DMA01}.  Note also that
the continuous category is \emph{more} rigid than the discrete one:
infinite divisibility, an assumption in
Theorem~\ref{thm:classification}, is automatic here (consistency of
$\Omega_{\sigma\sigma'} \stackrel{d}{=} \Omega_\sigma +
\Omega'_{\sigma'}$), so the classification needs no distributional
hypothesis beyond membership in the class.  Finally, the
most-singular-branch geometry (Corollary~\ref{cor:branch}) reads
natively: the probability that the cone above a point carries no
Poisson point down to scale $\ell$ is $\ell^{\,C}$---codimension
$C$, with no discretization anywhere.
\end{remark}

\section{Corollaries}
\label{sec:corollaries}

\begin{corollary}[Conservation constraint]\label{cor:conservation}
If there exists an index $k_0 > 0$ such that $\zeta_{k_0} = z_0$ for
a known constant $z_0$ fixed by an exact conservation law, then
\[
  \gamma = \frac{z_0 - C(1-\beta^{k_0/k})}{k_0}.
\]
This reduces the observable parameters from two to one.
\end{corollary}

\begin{corollary}[Codimension identification]\label{cor:codimension}
If the most singular structures have Hausdorff codimension
$C_{\mathrm{geom}}$ and $C = C_{\mathrm{geom}}$, then $\beta$
alone determines the full exponent curve, the multifractal spectrum,
and the cascade distribution.
\end{corollary}

\begin{corollary}[Spectrum width]\label{cor:width}
The width of the multifractal spectrum is
\[
  \Delta h = h_{\max} - h_{\min} = \frac{C}{k}\,|\ln\beta|,
\]
where $h_{\max} = \gamma + (C/k)|\ln\beta|$ (at $p = 0$, most
regular) and $h_{\min} = \gamma$ (at $p \to \infty$, most singular).
\end{corollary}

\begin{corollary}[Parameter-free stability constant]\label{cor:constant}
Since $|A| = C(1-\beta)\,|\ln r|$, the scale ratio cancels in
Theorem~\textup{\ref{thm:stability}}:
\[
  W_1\!\left(\frac{\eta}{\|\eta\|},\,\delta_\beta\right)
  \;\le\; \sqrt{\frac{1+\beta}{C(1-\beta)}}\;\sqrt{\varepsilon}
  \quad\textup{(i)},
  \qquad
  \le\; \sqrt{\frac{2}{C(1-\beta)}}\;\sqrt{\varepsilon}
  \quad\textup{(ii)},
\]
independent of $r$ and $k$ (and identical to the native continuous
constants of Corollary~\ref{cor:mrmstab}).  For fully developed
turbulence ($\beta = 2/3$, $C = 2$, in either the dissipation form
$k=1$ or the velocity form $k=3$) the constants are $\sqrt{5/2}
\approx 1.58$ and $\sqrt{3} \approx 1.73$: a measured violation
$\varepsilon$ of hierarchical symmetry confines the normalized
tilted jump measure within $1.58\sqrt{\varepsilon}$
(resp.\ $1.73\sqrt{\varepsilon}$) of $\delta_{2/3}$ in
Wasserstein-1 distance.
\end{corollary}

\section{Concluding remarks}
\label{sec:discussion}

The results of this paper show that the hierarchical symmetry~A1
carries considerably more force than might be expected from its
appearance as a simple linear recurrence.  Within i.i.d.\
multiplicative cascades it is equivalent to the log-Poisson class,
with sharp stability constants and the exact propagation rate
$\Theta(\sqrt\varepsilon)$; beyond independence it pins all
asymptotic statistics (and provably nothing more); and in the
continuous category it selects exactly the compound Poisson
cascade at the generator level, while no finite moment window of
structure functions can do so.  The following directions remain
open.

\begin{enumerate}
\item \textbf{Lattice-only finite-state rigidity.}
  Theorem~\ref{thm:finitestate} assumes the closed exponent form for
  all real $p \ge 0$; under lattice-only A1 the identity-theorem
  step is unavailable, and Carlson-type interpolation is blocked by
  possible complex eigenvalue crossings of the tilted transfer
  matrix.  We expect the conclusion to persist.

\item \textbf{Determination of~$k$, and the joint-in-$k$ test.}
  The hierarchy step~$k$ is treated as given; in applications it
  must be estimated.  A1 at several steps simultaneously imposes the
  compatibility constraint $\ln\beta(k) \propto k$, and the
  statistical gain from the joint test is unquantified.

\item \textbf{Boundary cases.}  The log-normal class is the $\beta
  \to 1$ closure point of the log-Poisson family ($b \to 0$,
  $\lambda b^2 \to \sigma^2$): quantifying the degeneration of
  identifiability as $\beta \uparrow 1$ would unify the
  classification with its principal rival.  The maximal-intermittency
  limit $\beta \to 0$ likewise deserves analysis.

\item \textbf{Statistics of the A1 test.}  The present results are
  exact-population statements.  A finite-sample theory---error bars
  on $(\hat\beta, \hat d^*)$, power against log-normal and
  log-stable alternatives, with the reading-(ii) constants of
  Theorem~\ref{thm:stability} and Corollary~\ref{cor:mrmstab} as the
  operative null band---is the missing link between the
  classification and data; the window obstruction of
  Theorem~\ref{thm:window} dictates that such a theory be built on
  magnitude statistics rather than high-order structure functions.
\end{enumerate}

\backmatter

\bmhead{Funding}
The author received no funding for this work.

\bmhead{Conflict of interest}
The author declares no competing interests.

\bmhead{Data availability}
Data sharing is not applicable to this article as no datasets were
generated or analyzed during the current study.


\end{document}